\documentclass[a4paper]{amsart}
\usepackage{amsfonts,amsmath,amssymb,amsthm,cmtiup}

\usepackage{mathrsfs}
\usepackage[noadjust]{cite}

\numberwithin{equation}{section}

\newtheorem{theorem}{Theorem}[section]

\newtheorem{corollary}[theorem]{Corollary}

\newtheorem{statement}[theorem]{Statement}

\theoremstyle{definition}
\newtheorem{remark}[theorem]{Remark}
\newtheorem{definition}[theorem]{Definition}
\newtheorem{construction}[theorem]{Construction}

\newtheorem{question}{Question}

\newcommand{\Ind}{\operatorname{Ind}}
\newcommand{\ind}{\operatorname{ind}}

\newcommand{\Int}{\operatorname{Int}}
\newcommand{\diam}{\operatorname{diam}}

\begin{document}

\author{I.~M.~Leibo}
\address{Moscow Center for Continuous Mathematical Education, Moscow, Russia}
\email{imleibo@mail.ru}

\title{Coincidence of Dimensions for Stratifiable Spaces}

\begin{abstract}
It is proved that $\Ind X = \dim X$  for any stratifiable space $X$.
This implies the coincidence of the dimensions $\ind$, $\Ind$, and $\dim$ for stratifiable spaces
with countable network. It is also proved that each stratifiable space is an S-space (has an S-network).
\end{abstract}



\maketitle

\section{Introduction}
\label{sec1}

Kat\v etov and Morita \cite{1} proved the equality $\dim X=\Ind X$ for any metric space $X$.
Yu.~M.~Smirnov and J.~Nagata proved that a space $X$ is metrizable if and only if there exists
a $\sigma$-locally finite base for the topology of $X$ (see \cite{2}). Each locally finite system of
sets in a space $X$ is closure-preserving in $X$ \cite{1}. Since the dimensions $\dim$ and
$\Ind$ coincide for any metrizable space, there arises the natural question of whether they coincide for
spaces with a $\sigma$-closure-preserving base. In this paper, we give a positive answer to  this
question.

In \cite{3}, Ceder introduced the notion of a space with a $\sigma$-closure-preserving base (an $M_1$-space); in the
same paper, he also introduced  spaces  with a $\sigma$-closure-preserving quasi-base ($M_2$-spaces) and $M_3$-spaces
(which Borges called stratifiable spaces \cite{4}).
It follows from the definitions of $M_1$-, $M_2$-, and $M_3$-spaces that every $M_1$-space is
an $M_2$-space and every $M_2$-space is stratifiable \cite{5}. A space with a $\sigma$-closure-preserving
base is not necessarily metrizable~\cite{5}.

Stratifiable spaces occupy a special place in the theory of generalized metric spaces.
They have many important and interesting properties. For example, countable products of
stratifiable spaces are stratifiable, closed images of stratifiable spaces are stratifiable, and
for stratifiable spaces the Dugundji--Borges extension theorem holds \cite[Theorem~4.3]{4}.
We will state other properties of stratifiable spaces as needed. A detailed
 description of the properties of stratifiable spaces can be found in the monograph~\cite{5}.

Clearly, the coincidence of $\dim$ and $\Ind$ for stratifiable spaces implies that
$\dim X= \Ind X$ for any space $X$ with a $\sigma$-closure-preserving base. This suggests the question
  of whether $\dim$ equals $\Ind$ for stratifiable spaces; it was asked in \cite{6, 7}.
Gruenhage and Junnila showed in \cite{5} that every stratifiable space has
a $\sigma$-closure-preserving quasi-base $\mathscr{B}$
(that is, is an $M_2$-space);  in a normal space, this quasi-base $\mathscr{B}$ can always be assumed
to be closed (consisting of closed sets). The question of whether the classes of $M_1$-spaces and of $M_2$-spaces
coincide remains open.

In this paper, we give a positive answer to the question of whether the dimensions $\dim$ and $\Ind$ coincide
for stratifiable spaces.

In \cite{2},  A.~V.~Arkhangel'skii asked what conditions must be added to the presence of a countable network
in a space $X$ for the equalities $\ind X=\dim X=\Ind X$ to hold. In \cite{8, 9}, the author
proved that $\dim X=\Ind X$ for any first-countable paracompact $\sigma$-space $X$.
This implies $\ind X=\dim X=\Ind X$ for any first-countable space with a countable network.
The equality of $\dim$ and $\Ind $ for stratifiable spaces implies $\ind X=\dim X=\Ind X$ for any
stratifiable space with a countable network.

According to Theorem~1 of \cite{10}, the classes of spaces with a $\sigma$-closure-preserving network,
with a $\sigma$-locally finite network, and with a $\sigma$-discrete network coincide. In this paper,
by a \emph{$\sigma$-space}  we mean a space with a closed $\sigma$-discrete network
(i.e., a $\sigma$-network consisting of closed sets), which was introduced and studied by A.~V.~Arkhangel'skii
in~\cite{2, 11}.

We use the following notation.
Given a metric space $M$ with metric $\rho$ and an $m \in M$, we set
\begin{gather*}
U_r(m)=\{x: x\in M, \rho(x,m)<r\},\\
B_r(m) = \overline{U_r(m)} = \{x: x\in M,\ \rho(x,m)\leq r\}.
\end{gather*}
For two systems of sets  $\nu=\{A_{\alpha}\in\Re\}$ and $\sigma = \{B_{\beta}\in{\L}\}$,
$$
\nu \wedge \varsigma = \{A_{\alpha} \cap B_{\beta}: A_{\alpha}\in\nu,\ B_{\beta}\in\varsigma\};
$$
by $\nu\vee\varsigma$  we denote the union of these systems, i.e.,
$$
\nu\vee\varsigma = \{A_\alpha: \alpha\in\Re\}\cup \{ B_\beta: \beta\in{\L}\}.
$$
For $A \subseteq M$ and $m\in M$, we set
$$
\rho(m,A) = \inf\{\rho(m,x), x\in A\}.
$$
We note at once that the function $f(m) = \rho(m,A)$, $m\in M$,  is continuous on $M$
for any fixed $A\subseteq M$
(see \cite[Chap.\,2, Problem~268]{12}).

We denote the interior of a set $B$ in a space $X$ by $\Int(B)$ and its closure by~$\overline{B}$.

In this paper, unless otherwise stated, all spaces under consideration are assumed to be normal $T_1$-spaces
of finite dimension $\Ind$ and all maps, continuous. We
use the terminology of \cite{1, 7, 13, 14}.

It is my pleasure to thank Professor A.~V.~Arkhangel'skii for attention and Professor O.~V.~Sipacheva
for useful discussions.

\section{Basic Definitions, Statements, and Constructions}
\label{sec2}

The definition and main properties of the dimensions $\ind$, $\Ind$, and $\dim$ can be found in~\cite{1}.

A system of sets  $\{\Upsilon_\alpha: \alpha\in A\}$ is said to be \emph{closure-preserving}
if, for any subset $A_1\subseteq A$,
we have $\overline{\bigcup\{\Upsilon_\alpha: \alpha\in A_1\}} = \bigcup\{\overline{\Upsilon_\alpha}:
\alpha\in{A_1}\}$. A system  of sets is \emph{$\sigma$-closure-preserving} if it is a countable union
of closure-preserving systems. If a $\sigma$-closure-preserving system of sets  in a space $X$ consists
of open sets and forms a base of $X$, then we say that $X$ is a \emph{space with a $\sigma$-closure-preserving
base}. Recall that a system $\mathscr{B}$ of not necessarily open sets in a space $X$ is called
a \emph{quasi-base} of $X$ if, for every point $x \in X$ and every  open
neighborhood $Ox$ of $X$, there exists a $B \in \mathscr {B}$ such that $x \in \Int(B)\subseteq B\subseteq Ox$.
If a system of sets is closure-preserving, then so is the system of closures of these sets.
In what follows, unless otherwise stated, we will consider closed closure-preserving systems of sets,
i.e., systems consisting of closed sets.

A continuous map $\varphi\colon X\to Y$ of a  space $X$ onto a space $Y$ is
called a \emph{condensation} if it is one-to-one (the inverse map is not necessarily continuous).

The standard definition of a stratifiable space can be found in \cite[Definition~1.1]{4}.
According to a theorem of Gruenhage and Junnila, a space is stratifiable if and only if it has
a $\sigma$-closure-preserving quasi-base. Therefore, in this paper, by a \emph{stratifiable space}
we mean a space with a $\sigma$-closure-preserving quasi-base.

Any stratifiable space is perfectly normal and paracompact, any subspace of a stratifiable
space is stratifiable space, and, as shown by Heath, any stratifiable space is
a $\sigma$-space \cite{5}.  Recall  that any countable product of paracompact
$\sigma$-spaces and any subspace of a paracompact $\sigma$-space are paracompact
$\sigma$-spaces~\cite{11}.

The following notion was introduced by the author in \cite{15} and, independently, by Oka in~\cite{6}.

  \begin{definition}
\label{def2.1}
We say that a closed network $\gamma=\{F_{\alpha}\}$ in a space $Y$ is an \emph{S-network}
if, for any closed set $F \subseteq Y$ and any open
neighborhood $OF$ of $F$, there exists a subsystem $\gamma(F)$ of $\gamma$ such that the union $\bigcup \gamma(F)$
of all elements of $\gamma(F)$ is closed in $Y$ and
\begin{equation}
\label{eq1}
Y \setminus OF\subseteq\bigcup\gamma(F)\subseteq Y\setminus F.
\end{equation}

We refer to a space having a $\sigma$-closure-preserving S-network as an
\emph{S-space}.
\end{definition}

As shown by the author in \cite[Corollary~1]{15}
(and, independently, by Oka in \cite[Theorem~4.3]{6}), the property of being an S-space
is a sufficient condition for the coincidence of the dimensions $\dim$ and
$\Ind$ in the class of stratifiable spaces.

It is unknown whether the property of being an S-space is a sufficient condition for the coincidence of the dimensions
$\dim$ and $\Ind$ in the class of paracompact $\sigma$-spaces.

Definition~\ref{def2.1} requires the fulfillment of \eqref{eq1} for every pair $(F, OF)$, where $F$ is an
arbitrary closed subset of $Y$ and $OF$ is any open neighborhood of $F$. The following purely technical
definition makes it possible to ``reduce the number'' of closed sets in Definition~\ref{def2.1}, i.e., to
concretize the closed sets $F \subseteq Y$; namely, it suffices to consider only closed sets belonging to some
$\sigma$-discrete network of~$X$.

 \begin{definition}[{\cite[Definition~5]{15}}]
\label{def2.2}
We say that a network $\gamma$ in a space $X$ is \emph{intrinsically special} if,
for every $F\in\gamma$ and every open  neighborhood $OF$ of $F$, there exists a subsystem $\gamma(F)$
of the network $\gamma$ such that
the union $\bigcup\gamma(F)$
is closed in $X$ and
 $$
X\setminus OF\subseteq\bigcup\gamma(F)\subseteq X\setminus F.
$$
\end{definition}

 Clearly, if a stratifiable space $X$ is an S-space, then any S-network of $X$ is intrinsically
special. The converse is also true: if a stratifiable space $X$ has an intrinsically special network $\gamma$,
then $X$ is an S-space by virtue of Theorem~4 and Corollary~1 in \cite{15}
(cf.\ conditions (2.1) in the proof of Theorem~2.3 in~\cite{14}).

 Therefore, if a stratifiable space $X$ has an intrinsically special network $\gamma$,
then it suffices to require the existence of a subsystem $\gamma(F)\subset \gamma$
satisfying condition \eqref{eq1} only for open neighborhoods of the form $OF=X\setminus G$,
where $G\in\gamma$ and $G \cap F=\varnothing$, rather than for
all open neighborhoods $OF$ of a closed subset $F$ of~$X$.

 Thus, to prove the equality $\dim X=\Ind X$ for a stratifiable space $X$, it suffices
to construct an intrinsically special network~$\gamma$  in~$X$.

 To this end, we need some definitions and facts.

 \begin{definition}[\cite{14}, \cite{15}]
 \label{def2.3}
 Given a space $Z$ of dimension $\Ind Z=n$, we say that a closed subset $F$ of $Z$ and its
neighborhood $OF$ \emph{determine the dimension} $\Ind Z$ of $Z$ if, for every neighborhood $U$ of $F$
such that $F \subseteq U \subseteq \overline{U} \subseteq OF$, we have $\Ind \operatorname{Fr} U \ge n-1$.
\end{definition}

\begin{definition}[\cite{7}, \cite{15}]
\label{def2.3}
We say that a set $\varphi=\{F_{\alpha}, OF_{\alpha}\}_{\alpha\in A}$ of closed subsets $F_{\alpha}$
 and their open neighborhoods $OF_{\alpha}$ in a space $Z$ is an \emph{everywhere f-system}
(\emph{an f-system}) if, for any (any closed) set $M \subseteq Z$, there exists an index $\alpha_0 \in A$
such that the pair
$\{M\cap F_{\alpha_0}, M\cap OF_{\alpha_0}\}$ determines the dimension~$\Ind M$.
\end{definition}

We refer to a space $X$ having an everywhere f-system as an \emph{everywhere f-space} and
to a space  having an f-system as an \emph{f-space}.

If the space $Z$ in Definition~\ref{def2.3} is only assumed to be normal, then we require the system
 $\{OF_{\alpha}\}_{\alpha\in A}$ of open sets to be $\sigma$-discrete in $Z$. If
$Z$ is a paracompact $\sigma$-space, then it suffices to require the system  $\{OF_{\alpha}\}_{\alpha \in A}$
to be $\sigma$-locally finite in~$Z$.

The point is that we need a set $\varphi=\{F_{\alpha}, OF_{\alpha}\}_{\alpha\in A}$ of closed subsets $F_{\alpha}$
and their open neighborhoods $OF_{\alpha}$  of $Z$ to prove the equality $\dim Z=\Ind Z$. If the space $Z$
it only assumed to be normal, then, to prove this equality, it is required (sufficient) that the system
$\{OF_{\alpha}\}_{\alpha\in A}$ of open sets  be $\sigma$-discrete in $Z$. But if
$Z$ is a paracompact $\sigma$-space, then it suffices that $\{OF_{\alpha}\}_{\alpha \in A}$ be only
$\sigma$-locally finite in $Z$ \cite[Lemma~1.11]{7}.

S.~Oka proved the following theorem.

\begin{theorem}[{\cite[Lemma~3.2]{6}}]
\label{th2.4}
Let $X$ be a paracompact $\sigma$-space, and let $\mathscr B = \bigcup\{\mathscr B_i:
i\in \mathbb N\}$ be a $\sigma$-closure-preserving system of closed sets in $X$ (for each $i=1,2,\dots$,
the system $\mathscr B_i$ of closed sets is closure-preserving in $X$). Then there exists a metric space
$Y_1$ and a compression $\psi\colon X\to Y_1$ such that $\psi(B)$ is closed in $Y_1$ for each
$B\in\mathscr B$ and the system $\psi(\mathscr B_i)=\{\psi(B): B\in \mathscr B_i\}$ is closure-preserving
in $Y_1$ for each $i=1,2\dots$\,.
\end{theorem}

Arkhangel'skii proved the following theorem:

\begin{theorem}[{\cite[Sec.~6]{11}}]
\label{th2.5}
Let $X$ be a paracompact $\sigma$-space, and let $\gamma =\{F_{\alpha}: \alpha \in A\}$ be a
$\sigma$-discrete network in $X$; suppose that $A=\bigcup_{i=1}^{\infty} A_{i}$ and
$\gamma_i=\{F_{\alpha}:\alpha\in A_{i}\}$ is a discrete system of closed subsets of $X$
for each $i=1,2,\dots$\,. Then there exists a condensation $\varphi\colon X\to Y$ of $X$ onto
a metric space $Y$ such that the set $\varphi(F_{\alpha})$ is closed in
$Y$ for every $\alpha\in A$ and the system $\varphi(\gamma_i)=\{\varphi(F_\alpha): \alpha\in A_i\}$
is discrete in $Y$ for each $i=1,2,\dots$,
i.e.,\ $\varphi(\gamma) = \{\varphi(F_\alpha):\alpha\in A\}$ is a closed
$\sigma$-discrete network of~$Y$.
\end{theorem}

Below we present preliminary constructions and related theorems. These constructions are needed
to construct and explicitly describe an intrinsically special network in a stratifiable space.

\begin{construction}
\label{cons2.6}
Let $X$ be a paracompact $\sigma$-space. Suppose that, for each $i\in \mathbb N$,  $\mathscr B_i$
is a closure-preserving system of closed sets in $X$ and let
$\mathscr B = \bigcup\{\mathscr B_i: i\in \mathbb N\}$ (then $\mathscr B$ is a $\sigma$-closure-preserving
system of closed sets in $X$). Suppose also that $\gamma =\{F_{\alpha}: \alpha \in A\}$
is a $\sigma$-discrete network in $X$, where
$A=\bigcup_{i=1}^{\infty} A_{i}$ and
$\gamma_i=\{F_{\alpha}:\alpha\in A_{i}\}$ is a discrete system of closed subsets of $X$
for each $i=1,2,\dots$\,. By Theorems~\ref{th2.4} and~\ref{th2.5}, there exist metric spaces $Y_1$ and $Y$
and condensations $\psi\colon X\to Y_1$ and $\varphi\colon X\to Y$ such that
$\psi(B)$ is closed in $Y_1$ for each
$B\in\mathscr B$, the system $\psi(\mathscr B_i)=\{\psi(B): B\in \mathscr B_i\}$ is closure-preserving in $Y_1$
for each $i=1,2\dots$, the set
$\varphi(F_{\alpha})$ closed in $Y$ for each $\alpha\in A$,
and the system $\varphi(\gamma_i)=\{\varphi(F_\alpha): \alpha\in A_i\}$ is discrete in $Y$.
Consider the diagonal map
$f=\psi\triangle\varphi\colon X\to M_0\subseteq(Y_1\times Y)$. Clearly, this is a condensation and
\begin{enumerate}
\item[(1)]
$f(\gamma)$ is a closed $\sigma$-discrete network in $M_0$;
\item[(2)]
$f(B)$ is closed in $M_0$ for each $B\in \mathscr B$ and the system $f(\mathscr B_i) = \{f(B): B\in \mathscr B_i\}$
is closure-preserving  in $M_0$ for each $i=1,2,\dots$\,.
\end{enumerate}

 Nagata proved an important and interesting theorem for metric spaces \cite[Theorem~V.6]{16}.
Below we state the part of it we need.

\begin{theorem}[{\cite[Theorem V.6]{16}}]
\label{th2.7}
Let $X$ be a metric space, and let $\tau$ be the topology on $X$ generated by its metric.
Then, on the metrizable space $(X,\tau)$, there exists a metric $\rho$
generating the same topology $\tau$ such that, for any real number $\varepsilon>0$, the system
$\{O_{\varepsilon}(x): x\in X\}$ of $\varepsilon$-neighborhoods of $x\in X$ in
the metric $\rho$, is closure-preserving  in $X$.
\end{theorem}

 The  metric $\rho$ mentioned in Theorem~\ref{th2.7} is called a \emph{Nagata metric}.

 Clearly, the system  $\{\overline{(O_{\varepsilon}(x))}: x\in X\}$ in Theorem~\ref{th2.7} is
closure-preserving in~$X$.
\end{construction}

 Heath~\cite{5} proved the following characterization of stratifiable spaces.

\begin{theorem}[\cite{5}]
\label{th2.8}
A space  $(X,\tau)$ is stratifiable if and only if
there exists a function $g\colon \mathbb N \times X \to  \tau$ such that
\begin{enumerate}
\item[\rm (1)] $\{x\} = \bigcap\{g_n(x): n=1,2,\dots \}$ for every $x\in X$;
\item[\rm (2)] if $x\in g_n(x_n)$ for $n=1,2,\dots$\,, then the sequence $(x_n)_{n\in \mathbb M}$ converges
to~$x$;
\item[\rm (3)] if $H$ is a closed set in $X$ and $y\notin H$, then there exists an $n\in \mathbb N$ such that
  $y\notin\overline{\bigcup\{g_n(x):x\in H\}}$;
\item[\rm (4)] if $y\in g_n(x)$, then $g_n(y)\subseteq g_n(x)$.
\end{enumerate}
\end{theorem}

Recall that any stratifiable space $X$ is a hereditarily stratifiable paracompact
  $\sigma$-space~\cite{5}.

The following construction is an adaptation of Construction~\ref{cons2.6} to the case we need.

\begin{construction}
\label{cons2.9}
Let $X$ be a stratifiable space. Since $X$ is a $\sigma$-space, it has a closed
$\sigma$-discrete network $\gamma=\{F_{\alpha}: \alpha\in A\}$, where $A=\bigcup_{i=1}^{\infty} A_{i}\}$ and
  $\gamma_i=\{F_{\alpha}:\alpha\in A_{i}\}$ is a discrete system in $X$ for each $i=1,2,\dots$\,.
Moreover, according to \cite{5}, being stratifiable, the space $X$ has a closed
$\sigma$-closure-preserving quasi-base $\mathscr{B} = \bigcup\{\mathscr B_n: n=1,2,\dots \}$,
where $\mathscr B_n = \{B_{\alpha}: \alpha \in A_n\}$ is a closure-preserving
system in $X$ for $n=1,2,\dots$ and all sets $B_{\alpha}$ are closed in $X$.

Recall that, in Construction~\ref{cons2.6}, we defined the diagonal map $f=\psi \triangle \varphi$,
where $\psi$ and $\varphi$ are the same condensations as in Theorems~\ref{th2.4} and~\ref{th2.5}.

For $x\in X$ and $i\in \mathbb N$, we set $g_i(x) = X\setminus \bigcup \{B\in \mathscr B_i:
  x\notin B\}$.  The set $\bigcup \{B \in \mathscr B_i:
  x\notin B\}$ is closed in $X$, because  the system $\mathscr B_i$ is closure-preserving; therefore,
$g_i(x)$ is open in $X$. Since
  $f(\bigcup \{B \in \mathscr B_i: x\notin B\})$ is closed in $M_0$ and $f$ is a condensation, it follows that
 $f(g_i(x))$ is open in $M_0$.
The system $\{g_i(x): x\in X,\ i=1,2,\dots\}$ of open sets satisfies all assumptions of
Theorem~\ref{th2.8}.

Given any set $F\subseteq X$, by $O_i F$ we denote the open set $\bigcup \{g_i(x): x\in F\}$.

The diagonal map $f$ has the following properties:

\begin{enumerate}
\item[(1)]
$f(\gamma)=\{f(F): F\in \gamma\}$ is a closed $\sigma$-discrete network in $M_0$; we denote it by $\gamma^1$;
\item[(2)]
 $f(B)$ closed in $M_0$ for each $B\in \mathscr{B}$, and the system $f(\mathscr B_i)= \{f(B):B\in \mathscr B_i\}$
is closure-preserving in $M_0$ for each $i=1,2\dots$;
\item[(3)]
the set $f(g_i(x))=g^1_i(f(x))$ is open in $M_0$ for any $x\in X$ and $i=1,2,\dots$,
 because the system $\mathscr B_i$ closure-preserving in $X$; therefore, the set $f(O_i F)$
is open in $M_0$ for any $i\in \mathbb N$ and $F\subset X$.
\end{enumerate}

This completes Construction~\ref{cons2.9}.
\end{construction}

Let $\rho$ be a Nagata metric on the metric space $M_0$ in Construction~\ref{cons2.9}
(in the rest of the paper, we assume the space $M_0$ in Construction~\ref{cons2.9} to be
endowed with this metric~$\rho$).

For each $i\in \mathbb N$, we fix the closed closure-preserving
cover $\eta_i = \{\overline{O_{1/i}(x)}: x\in X\}$ of the space $M_0$.
The system $\eta = \bigcup\{\eta_i : i \in \mathbb N\}$ is a
  $\sigma$-closure-preserving network in $M_0$, and we can apply Statement~1.9 of \cite{14}. This
  statement is contained in the proof of Theorem~1 in the paper \cite{10} by Siwiec and Nagata,
although it is not explicitly stated there. Below we formulate it a form convenient for our purposes.

\begin{statement}
\label{stat2.10}
Suppose that a normal space $Y$ has a $\sigma$-closure-preserving network
$\vartheta=\{\Phi_{\alpha}: \alpha\in \mathscr L\}$, where $\mathscr L=\bigcup_{i=1}^{\infty} \mathscr L_{i}$ and
the system $\{\Phi_{\alpha}: \alpha\in \mathscr L_{i}\}$ is closure-preserving in $Y$ for each $i=1,2, \dots$\,.
Then $Y$ has a $\sigma$-discrete network
$\mu=\{W_{\beta}: \beta\in \Gamma\}$, where $\Gamma=\bigcup_{i=1}^{\infty} \Gamma_{i}$ and the system
  $\{W_{\beta}: \beta\in \Gamma_{i}\}$ is discrete in $Y$ for each $i=1,2, \dots$, and any set
  $\mathscr L^{1} \subseteq\mathscr L$ of indices contains a subset
  $\Gamma^{1} \subseteq \Gamma$ such that
  $\bigcup \{\Phi_{\alpha}: \alpha\in \mathscr L^{1}\}=\bigcup \{W_{\beta}: \beta\in \Gamma^{1}\}$.
\end{statement}

\subsection*{($\boldsymbol\ast$)  Construction of an intrinsically special network}
We return to Construction~\ref{cons2.9}. Consider
  the metric space $M_0$. In the case under consideration, the network $\vartheta$ in Statement~\ref{stat2.10}
is the network $\eta$ of $M_0$ defined above. According to Statement~\ref{stat2.10},
$M_0$ has a $\sigma$-discrete network $\mu$ such that every element of $\eta$ is the union
  of some elements of $\mu$. Consider the new closed $\sigma$-discrete network
$\gamma_2$ = $\mu \wedge \gamma^1$ in $M_0$,
where $\gamma^1 = f(\gamma)=\{f(F):F\in \gamma\}$ (see\ property (1) in Construction~\ref{cons2.9}).
Each closed element of the $\sigma$-closure-preserving network $\eta$ is the union of some
  elements of $\gamma_2$. We denote the preimage of
  $\gamma_2$ under
mapping $f$  by $\gamma_0$, i.e., set
$$
\gamma_0 = f^{-1}(\gamma_2)=\{f^{-1}(F): F\in \gamma_2\}.
$$
This is a $\sigma$-discrete network in the space $X$; therefore,
  $\gamma_0 = \bigcup\{\gamma^0_i:  i\in \mathbb N\}$, where $\gamma^0_i$ is a discrete system in $X$
for each $i\in \mathbb N$.
In what follows, we show that the $\sigma$-discrete network $\gamma_0$ is an intrinsically special
network of the space~$X$.

 To show this, we somewhat change the initial closed $\sigma$-closure-preserving quasi-base
  $\mathscr{B} = \bigcup\{\mathscr B_i : i\in \mathbb N\}$ in the space~$X$.

We leave the condensation $f\colon
X\to  M_0$ constructed above in Construction~\ref{cons2.9} intact.
We change only the neighborhoods $g_i(x)$ in $X$.

 Namely, we proceed as follows.

 1. Consider the open covers $\zeta_k = \{O_{1/2k}(y):
  k\in \mathbb N,\ y\in M_0\}$ of the metric space $M_0$. For each open
 cover $\zeta_k$, we choose an open locally finite refinement $\omega_k = \{O_{\beta}: \beta \in A_k\}$
and set
 $$
\overline\omega_k = \{\overline O_{\beta}: \beta \in A_k\}.
$$
Note that, for each
  $\overline O_{\beta} \in \overline\omega_k$, we have
$\diam(\overline O_{\beta})\leq 1/k$.
 Consider the system $f(\mathscr B_k)=\{f(F): F\in \mathscr B_k\}$, $k \in \mathbb N$. This is a
closure-preserving system of closed sets  in $M_0$ (see property (2) in Construction~\ref{cons2.9}),
and the system $\overline\omega_k$ of closed sets is locally finite in $M_0$. Let $\kappa$ denote
  the system $f(\mathscr B_k)\wedge \overline\omega_k$ of closed sets in $M_0$. Note that $\kappa$
   is closure-preserving in $M_0$, because the intersection of any closed set with a closed
   closure-preserving system is closure-preserving and any locally finite system is hereditarily locally finite,
i.e., hereditarily closure-preserving. If $G\in \kappa$, then $\diam(G)\leq 1/k$.
 We set
$$
\mathscr B^0_k = f^{-1}(\kappa)=\{f^{-1}(F): F\in \kappa\}.
$$
Clearly, $\mathscr{B}^0 = \bigcup\{B^0_k : k\in \mathbb N\}$ is a closed
  $\sigma$-closure-preserving quasi-base in the space $X$.

2. We want to show that the $\sigma$-discrete network $\gamma_0$ in $X$ is as required, i.e.,
this is an intrinsically special network in the space $X$.
 Setting $\mathscr B^1_i = \mathscr B^0_i \vee \gamma^0_i$, we obtain the new $\sigma$-closure-preserving
quasi-base $\mathscr{B}^1  = \bigcup\{\mathscr B^1_i: i\in \mathscr N\}$ in $X$.
Note that property (2) of the condensation $f\colon X\to  M_0$ in Construction~\ref{cons2.6} holds not only
for the quasi-base $\mathscr{B}$, but also for $\mathscr{B}^1$, i.e., $f(B)$ is closed in $M_0$ for each
$B\in \mathscr{B}^1$ and the system
$f(\mathscr B_i^1)= \{f(B):B\in \mathscr B_i^1\}$ is closure-preserving in $M_0$ for each $i=1,2\dots$.

 ($\ast$$\ast$) The quasi-base $\mathscr{B}^1$ has a new convenient property; namely,
if $F\in \gamma^0_i$,  $H\subseteq X$ is closed, and $F\cap H = \varnothing$, then
$O_i H\cap F = \varnothing$, where $F \in \mathscr{B}^1_i$ and $O_i H$ is defined as in Construction~\ref{cons2.9} but for
$g_i(x) = X\setminus \cup \{B:  B \in \mathscr B^1_i,
 x\notin B\}$. The sets $g_i(x)$ are open in $X$, and the sets $f(g_i(x))$ are open in~$M_0$.

 Let us summarize.

\subsection{Properties of the constructed objects}
\label{subsec2.1}

\textbf{1.}\enspace
 We have a stratifiable space $X$. It has a closed $\sigma$-discrete network $\gamma$ and a closed
 $\sigma$-closure-preserving quasi-base $\mathscr{B}$. We also have a condensation $f\colon X\to  M_0$
 of the space $X$ onto a metric space $M_0$  (see\ Construction~\ref{cons2.9}) such that
$f(\gamma)$ is a closed $\sigma$-discrete network in $M_0$ and
$f(\mathscr{B})$ is a closed $\sigma$-closure-preserving system in $M_0$. On the metric space $M_0$,
we fix a Nagata metric $\rho$ and the countable system $\{\eta_i = \overline{O_{1/i}(x)}: i \in \mathbb N\}$
of closed closure-preserving covers whose union $\eta = \bigcup\{\eta_i = \overline{O_{1/i}(x)}:
i \in \mathbb N\}$ forms a network in $M_0$. Inn what follows,  $X$, $M_0$, $f$, $\rho$, and $\eta$ are
assumed to be fixed.

\textbf{2.}\enspace
On the space $X$, there arises a new closed $\sigma$-discrete network $\gamma_0 =
\bigcup\{\gamma^0_i:
 i\in \mathbb N\}$, where all $\gamma^0_i$ are discrete systems in $X$ (see $(\ast)$).
Moreover, $f(\gamma_0)$ is a closed $\sigma$-discrete network
in $M_0$ and every element of $\eta$ (i.e., every $\overline{O_{1/i}(x)}$) is a union of elements
of the system~$f(\gamma_0)$.

\textbf{3.}\enspace
Using the $X$ closed $\sigma$-closure-preserving quasi-base  $\mathscr{B}$ of $X$, we constructed a new closed
 $\sigma$-closure-preserving quasi-base
$\mathscr{B}^0 = \bigcup\{\mathscr B^0_k : k\in \mathbb N\}$, where all $\mathscr B^0_k$ are closure-preserving.
This  $\sigma$-closure-preserving quasi-base has the important property that,
for each $B\in \mathscr B^0_k$,  $\diam f(B)\leq 1/k$ (see\ ($\ast$), item~1).

\textbf{4.}\enspace
To the quasi-base $\mathscr{B}^0$ we added the elements of the network $\gamma_0$, setting $\mathscr B^1_i =
 B^0_i \vee \gamma^0_i$. As a result, we obtained the closed $\sigma$-closure-preserving quasi-base $\mathscr{B}^1 =
\{\mathscr B^1_i
 : i\in \mathbb N\}$. For $i\in \mathbb N$, $x\in X$, and $H\subset X$ we set
\begin{gather*}
g_i(x) =
 X\setminus \cup \{B \in \mathscr B^1_i: x\notin B\},\\
O_iH= \bigcup\{g_i(x): x\in H\}.
\end{gather*}

Note that the system of open sets $g_i(x)$, $x\in X$, $i=1,2,\dots$,  satisfies all assumptions
of Theorem~\ref{th2.8}. For any $B\in \mathscr B^1_i$ and any closed set $H\subset X$ disjoint from $B$
 (in particular, for $F\in\gamma^0_i$; see ($\ast$$\ast$)),
we have $O_i H \cap B = \varnothing$.
The sets $g_i(x)$ are open in $X$, and the sets $f(g_i(x))$ are open in $M_0$, because the $\mathscr B^1_i$
are closed closure-preserving systems in $X$,
 $f(\mathscr B^1_i)$ is a closed closure-preserving system in $M_0$, and $f$ is a condensation
(for $x\in X$, $\bigcup\{B\in \mathscr B^1_i: x\notin B\}$ is closed in $X$ and
$f(\bigcup\{B\in \mathscr B^1_i: x\notin B\})$ is closed in~$M_0$).

 Now we can show that $\gamma_0$ is an intrinsically special network in $X$ and, as a consequence, $X$ is
 an S-space.

 \section{Main Theorems}
\label{sec3}

 The main results of this paper are Theorems~\ref{th3.1}, \ref{th3.6}, and \ref{th3.7} and
Corollary~\ref{cor3.8}.

\begin{theorem}
\label{th3.1}
Any stratifiable space $X$  has a closed $\sigma$-discrete
 intrinsically special network \textup(see\ Definition~\ref{def2.2}\textup).
\end{theorem}

  In what follows, we will show  that the closed $\sigma$-discrete network
$\gamma_0 = \bigcup\{\gamma^0_i: i\in \mathbb N\}$ constructed above, where all $\gamma^0_i$ are discrete in $X$,
is intrinsically special. To this end, we must prove that, for
 any $F\in \gamma_0$ and any closed subset $\Phi$ of $X$ disjoint from $F$,
 there exists a subsystem $\gamma_0(F)$ of the network $\gamma_0$ such that the union
 $\bigcup \gamma_0(F)$ of its elements is closed in $X$, $\Phi \subseteq \bigcup \gamma_0(F)$,  and
 $\bigcup \gamma_0(F) \cap F= \varnothing$.

 Consider the condensation $f\colon X\to  M_0$ in Construction~\ref{cons2.9}. Take $F\in \gamma_0$
and a closed subset $\Phi$ of $X$ disjoint from $F$. Take  the closed
$\sigma$-closure-preserving quasi-base
$\mathscr{B}^1 = \{\mathscr B^1_i : i\in \mathbb N\}$ of $X$ specified in item~4 of Subsection~2.1.
 Let $f(\Phi) = S$ (the set $S$ is not necessarily closed in $M_0$). It follows from
property~($\ast$$\ast$) that there exists a number
$n_0 \in \mathbb N$ for which
 $O_{n_0} \Phi \cap F=\varnothing$. Indeed, $F\in \gamma^0_j$ for some $j\in \mathbb N$, and we can
take $n_0 =j$. We have $F \in \mathscr B^1_{n_0}$; therefore, $O_{n_0} \Phi \cap F=\varnothing$,
 where $O_{n_0} \Phi=\bigcup \{g_{n_0}(x): x\in \Phi\}$ and
  $g_{n_0}(x) = X\setminus \bigcup \{B \in \mathscr B^1_{n_0}: x\notin B\}$ are open
neighborhoods of points $x\in \Phi$ in $X$. The system $\mathscr B^1_{n_0}$ is closure-preserving in $X$,
  and the system $f(\mathscr B^1_{n_0})$ is closure-preserving in $M_0$. Therefore,
   the sets $f(g_{n_0}(x))$ are open in $M_0$ (and hence $f(O_{n_0} \Phi) = O_{n_0}S$ is open in $M_0$).
 The set $O_{n_0} \Phi$ is open in $X$. Let $T$
  denote the set $M_0 \setminus O_{n_0}S$. The set $O_{n_0}S$ is open in $M_0$,
and the set $T$ is closed in~$M_0$.

 Let us perform several constructions  necessary to prove Theorem~\ref{th3.1}.

   \begin{construction}
\label{cons3.2}
 Let us partition $S$ into subsets $S_i$, $i=1,2,\dots$,  as follows:
\begin{align*}
S_1 &= \{s\in S:  \rho(s,T)\geq 1\},\\
S_i &= \{s\in S:  1/i \leq \rho(s,T)\leq 1/(i-1), i=2,3,4,\dots \}.
\end{align*}
Consider the ``strips'' of $O_{n_0}S$ defined by
\begin{align*}
\Pi_1&=\{m\in S: \rho(m,T) \geq1\},\\
\Pi_i&= \{y\in O_{n_0}S:  1/i\leq \rho(y,T)\leq 1/(i-1),\ i=2,3,4,\dots \}.
\end{align*}
Note that $S_i \subseteq \Pi_i$ for $i=1,2,3,\dots$\,. The system of sets
 $\{S_i: i\in \mathbb N\}$ is locally finite in  $O_{n_0}S$, because so is the system of strips.
Let $\overline{S} \cap T=L$. We assume that all sets $S_i$ are nonempty.
Take any $i\in \mathbb N$ and consider the set $S_i$. For  $s\in S_i$,
choose closed neighborhoods $B_{r_i}(s)$ of the same radius $r_i$ such that
$0< r_i < 1/2 i^2$; to be more specific, we take the least $p_i \in \mathbb N$ for which $1/p_i<1/2 i^2$ and
set $r_i=1/{p_i}$. Let $F_i = \bigcup \{B_{1/{p_i}}(s): s\in S_i\}$. Note that $S_i \subseteq F_i$ and
 $B_{1/{p_i}}(s)\in \eta_{p_i}$. The set $F_i$ is closed in $M_0$ by virtue of Nagata's Theorem~\ref{th2.7},
because all $B_{1/{p_i}}(s)$, $s\in S_i$, have the same radius for fixed $i\in \mathbb N$.
Thus, every set
$F_i$ is the union of some elements of the closed $\sigma$-discrete network $\gamma_2$ (see\
$(\ast)$).  Therefore, $f^{-1}(F_i)$ is the union of some subsystem of the system $\gamma_0$ in $X$
(more specifically, of some subsystem of the system $\gamma^0_{p_i})$.
  By Nagata's Theorem~\ref{th2.7}, the set $F_i$ is closed in $M_0$; therefore, $f^{-1} F_i$ is closed in $X$.
  Note that, by construction, the system $\{F_i: i\in \mathbb N\}$ of closed sets in $M_0$ is  locally finite in
$O_{n_0}S$   (by construction, it follows from the inequality $1/p_i<1/2 i^2$ that each $F_i$, $i=1,2,\dots$,
intersect only finitely many
``neighboring'' sets~$\Pi_j$)
    and $F_i \cap T=\varnothing$ for $i=1,2,\dots$\,. Let $\mathfrak{F}$ denote the  set
  $\bigcup \{F_i: i\in \mathbb N\}$. Since $r_i=1/p_i \to  0$ as $i \to \infty$, it follows that
  $\overline{\mathfrak{F}}\cap T =L= \overline{S} \cap T$; moreover, $\mathfrak{F}$ is closed in
  $O_{n_0}S$.
\end{construction}

   Note at once that if the set $S$ is closed in $M_0$ or
  $\overline{S} \cap f(F)= \varnothing$ for some $F\in \gamma_0$, then, obviously, there exists a subsystem
of $\eta$ whose union is closed in $M_0$ and disjoint from~$f(F)$.

Indeed, the system $\eta = \bigcup\{\eta_i : i \in \mathbb N\}$,
where $\eta_i = \{\overline{O_{1/i}(x)}: x\in X\}$, is a
  $\sigma$-closure-preserving network in $M_0$. Therefore, the union of all elements of $\eta$
intersecting $\overline{S}$ but disjoint from $f(F)$ is the required closed set in $M_0$ (because
 $\rho(s,f(F))>0$ and $\diam O_{1/i}(x) \to 0$ as $i\to \infty$ for each point $s\in \overline{S}$).

Clearly, this union of elements of $\eta$ is the union some elements of $f(\gamma_0)$
(see item~2 in Section~\ref{subsec2.1}).

   \begin{remark}
\label{rem3.3}
 Let us explain what purpose the closed sets  $\{F_i: i\in \mathbb N\}$ serve.
In what follows, we will show that the
 sets $F_i$ and $S_i \subseteq F_i$, $i=1,2,\dots$, can be chosen so that
 $f^{-1}(\mathfrak{F})$ is closed in $X$. Since each $f^{-1} F_i$ is the union of some subsystem
 of $\gamma_0$, it follows that $f^{-1}(\mathfrak{F})$ is the closed union of some subsystem of
 $\gamma_0$  and $\Phi \subseteq f^{-1}(\mathfrak{F})$.
  But $f^{-1}(\mathfrak{F})\cap F= \varnothing$, where $F\in \gamma_0$, and by construction
$F\in f^{-1}(T)$. Therefore, $\gamma_0$ is an intrinsically special network.
\end{remark}

  \begin{construction}
\label{cons3.4}
 Let us denote $f^{-1}L$ by $L_1 \subseteq X$ (the set $L=\overline{\mathfrak{F}}\cap T =
\overline{S} \cap T$ was defined at the end of Construction~\ref{cons3.2}). Take a point $x$ in $L_1$.
Let us show
 that the sets $F_i$, $i=1,2,\dots$, can be chosen  so that $x\notin \overline{f^{-1}(\mathfrak{F})}$.

  Note that $x\notin \Phi$, because $x \in L_1$. Hence there exists an element $B$ of the closed
 $\sigma$-closure-preserving quasi-base $\mathscr{B}^0 = \{B^0_k : k\in \mathbb N\}$ such that
$x\in \Int(B) \subseteq B$, $B\cap \Phi=\varnothing$, and, therefore, $f(B) \cap S= \varnothing$.
  Let $B\in B^0_k$. Then  $\diam f(B) \leq 1/k$
(see\ item~3 in Section~\ref{subsec2.1}). We set $f(x)=y\in L$
and $f(B)=P$; then $P$ is closed in $M_0$. Let us estimate the distance from a point of $P$ to $S_i$.
Take $y_i \in S_i$ and $z\in P$. The triangle inequality for the points $y$, $z$, and $y_i$ in the
metric space $M_0$ implies $\rho(y_i,z)\geq \mod(\rho(y_i,y) - \rho(y,z))$. By the definition of the distance
from a point to a closed set, we have $\rho(y_i,y)\geq 1/i$, and by the definition of diameter, we have
$\rho(y,z)\leq 1/k$. Therefore, $\rho(y_i,z) \geq \mod(1/i - 1/k)$. Here $k$ is a fixed positive integer
and $i=1,2,3,\dots$\,.
It follows that  $\rho(P,S_i)=d_i>0$ for all $i\neq k$; $i=k$, it may happen that $\rho(P,S_i)=0$.
For $i\neq k$, we define $\{F_i: i\neq k\}$ as in Construction~\ref{cons3.2}, choosing
$p_i \in \mathbb N$ so that $1/p_i< \min(1/2 i^2, d_i)$. Then $F_i \cap P=\varnothing$ for $i\neq k$.
For $i=k$, we define a closed set $F_k$ precisely as described in Construction~\ref{cons3.2}
(and choose $p_i$ so that  $1/p_i < 1/2i^2$).
This closed set $F_k$ is the union of some subsystem of the closed $\sigma$-discrete network $\gamma_2$
(see\ $(\ast)$).
 This is the only $F_i$ which may intersect $P$. However, as mentioned in Construction~\ref{cons3.2},
we have $F_k \cap T = \varnothing$; therefore, $y \notin F_k$ and $x\notin f^{-1}(F_k)$, which implies the
existence of an open  neighborhood $Ux\subseteq X$ of $x$ for which $Ux \cap f^{-1}(F_k)=\varnothing$. We set
$Vx=Ux \cap \Int(B)$. Since $F_i \cap P=\varnothing$ for
  $i\neq k$ and  $x\in \Int(B)$, it follows that $Vx\cap f^{-1}(\mathfrak{F}) =
  \varnothing$. Therefore, $x\notin \overline{f^{-1}(\mathfrak{F})}$.
\end{construction}

  The set $f^{-1}(\mathfrak{F})$ is the union of a subsystem of the closed
  $\sigma$-discrete network $\gamma_0$.

  Now we must define sets $F_i$, $i=1,2,\dots$, so that
$L_1$ and $\overline{f^{-1}(\mathfrak{F})}$ be disjoint.

\begin{construction}
\label{cons3.5}
As in Construction~\ref{cons3.4}, take any point $x \in L_1$. Since
$x\notin \Phi$, it follows that there exists an element $B_x$ of the closed $\sigma$-closure-preserving quasi-base
$\mathscr{B}^0 = \bigcup\{\mathscr B^0_k: k=1,2,\dots \}$ such  that  $x\in \Int(B_x)\subseteq B_x$
and $B_x \cap \Phi=\varnothing$.

We denote the closed set $\bigcup \{B_x:  B_x \in B^0_k\}$ by $Q_k$. Note that
  $L_1 \subseteq \bigcup\{\Int Q_k:  k=1,2,\dots \}$.
Take a $k\in \mathbb N$ for which $B_x \in \mathscr B^0_k$. We have
 $\diam f(B_x)\leq 1/k$. Let us denote the closed set $f(Q_k)$ by $P_k$. Recall that $L=f(L_1)$.

The sets $P_k$, $k=1,2,\dots$, will play the role of the set $P$ in Construction~\ref{cons3.4} in turns.
Some of these sets may be empty; for convenience, we will not
throw them out.

 Consider the sets $S_i$ in Construction~\ref{cons3.2}.
We associate each $S_i$, $i=1,2,\dots$, with the set $P_i$ (possibly empty). We construct sets $F_i$
  by induction. Note again that  the sets $Q_k$ and,
accordingly, $P_k$ may be empty for some $k\in \mathbb N$. We will take this into account in our construction;
 namely, if $P_k = \varnothing$, then we will assume that
 $\rho(S_i, P_k)=d_{ik} = {+}\infty$, i.e., that $\rho(S_i, P_k)=d_{ik}>N$ for  any
 positive integer $N$ (this somewhat artificial assumption is needed to obtain
 correct minima of finite sequences containing ${+}\infty$ and at least
one real number).
  Recall (see\ Construction~\ref{cons3.4}) that $\rho(S_i, P_k)=d_{ik}>0$ for any $k, i\in \mathbb N$,
possibly except in the  case
   $k=i$, in which $\rho(S_i, P_k)$ may equal zero. We argue by induction on $i\in \mathbb N$.

 Let $i=1$.

 Regardless of whether  the set $P_1$ is empty or not, we choose $p_1$ so that
 $1/p_1 < 1/(2\cdot 1^2)$ (i.e.,\ $p_1>2$) and define the closed set $F_1$ as in
Construction~\ref{cons3.2}. The set $F_1$ is the union of a subsystem of the closed
 $\sigma$-discrete network $\gamma_2$ (see~($\ast$)).

 Let  $i=2$.

 We construct a closed subset $F_2$ of $M_0$  as follows. We set
 $\rho(S_2 , P_1)=d_{2,1}$, where $d_{2,1}>0$ (if $P_1$ is empty, then $d_{2,1}={+}\infty$). We choose
 $p_2 \in \mathbb N$ so that
$1/p_2 < \min(1/2 \cdot 2^2, d_{2,1})$ (see\ Construction~\ref{cons3.4}).
Now we define $F_2$ as in Construction~\ref{cons3.2}. Note that $F_2 \cap P_1=\varnothing$
 ($F_2$ may intersect $P_2$ if $P_2$ is nonempty). By construction, $F_1$ may intersect both
 $P_1$ and with $P_2$, but $F_2$ may intersect only $P_2$. Each of the closed sets $F_1$ and $F_2$
is the union of a subsystem of the closed $\sigma$-discrete network $\gamma_2$ (see\ $(\ast)$ and
Construction~\ref{cons3.2}).
Moreover, $S_1 \subseteq F_1$ and  $S_2 \subseteq F_2$  (the set $S_i$ are the same as in
Construction~\ref{cons3.2}), and the sets $F_1$ and $F_2$ are disjoint from the closed set~$T$.

The induction hypothesis:

 Suppose that, for $i=1,2,\dots, n$, sets $F_1$, $F_2$, \dots, $F_n$ are constructed and have
the following properties:

\begin{enumerate}
\item[(a)]
$S_i \subseteq F_i$, $T \cap F_i = \varnothing$, and $F_i$ is closed in $M_0$.
  Each set $F_i$, $i=1,2,\dots ,n$, is the union of a subsystem of the closed $\sigma$-discrete
 network $\gamma_2$ (see\ $(\ast)$ and Construction~\ref{cons3.2}).
\item[(b)]
 For each fixed $j_0$, $1\leq j_0<n$, the closed set $P_{j_0}$
 may intersect only those $F_j$ for which $j\geq j_0$. If
 $j_0 <j \leq n$, then $P_{j_0}\cap F_j=\varnothing$. If
$j_0 =n$, then $P_n$ may intersect any of the sets
  $F_j$, $1\leq j\leq n$.
\end{enumerate}

By construction, the sets $F_1$ and $F_2$ have properties (a) and (b).

 Let  $i=n+1$.

 We construct a closed set $F_{n+1}$ as follows. The set $F_{n+1}$ will not depend on the sets $F_1, F_2, \dots, F_n$.
It depends on $S_{n+1}$ and the sets $P_k$, where $1 \leq k \leq n$.
  Take $S_{n+1}$. For $k \leq n$, we set
 $\rho(S_{n+1}, P_k)=d_{{n+1},k}$; then $d_{{n+1},k}>0$
(this is proved in Construction~\ref{cons3.4}) and $d_{{n+1},k}={+}\infty$ if $P_k$ is empty.
Choose $p_{n+1} \in \mathbb N$ so that
  $1/p_{n+1} < \min(1/2(n+1)^2, \{d_{{n+1},k}): k=1,2,\dots ,n\})$.
Next,  for each point $s\in S_{n+1}$,  we take
 the closed neighborhood $B_{1/{p_{n+1}}}(s)\in \eta_{p_{n+1}}$
and define $F_{n+1}$ as in Construction~\ref{cons3.2}, i.e., put
$F_{n+1} = \bigcup \{B_{1/{p_{n+1}}}(s): s\in S_{n+1}\}$. The set $F_{n+1}$
is closed by Theorem~\ref{th2.7}, and it is
the union of a subsystem of the closed
 $\sigma$-discrete network $\gamma_2$ (see\ $(\ast)$ and Construction~\ref{cons3.2}).

  By construction, the following conditions hold:

\begin{enumerate}
\item[(i)]
$S_{n+1} \subseteq F_{n+1}$, $F_{n+1} \cap T = \varnothing$, the set $F_{n+1}$ is closed in
 $M_0$, and  $F_{n+1}$ is the union of a subsystem of the
closed
 $\sigma$-discrete network $\gamma_2$;
\item[(ii)]
the closed subset  $F_{n+1}$ of  $M_0$  may intersect $P_{n+1}$, but $F_{n+1} \cap P_k =
 \varnothing$ for $k=1,2,\dots ,n$;  this follows from the choice of the number $1/p_{n+1}$.
\end{enumerate}

  Therefore, by construction, the closed sets $F_1$, $F_2$, \dots,
$F_n$, $F_{n+1}$ have properties (a) and~(b). Next, we construct the set $F_{n+2}$ in a similar way.

 Thus,  we have constructed the  system of closed subsets $F_i$ of $M_0$,
$i=1,2,\dots$, which is locally finite in $O_{n_0}S$. We have
  $S_i \subseteq F_i$ and $T \cap F_i = \varnothing$, and each set $F_i$ is the union of some
 subsystem of the closed $\sigma$-discrete network $\gamma_2$; moreover,
  each closed set $P_i$,  $i=1,2,\dots$, may have nonempty intersections with only finitely many
sets $F_j$, namely, with those $F_j$ for which $j\leq i$.

  As in Construction~\ref{cons3.2}, we set $\mathfrak{F} = \bigcup \{F_i: i\in \mathbb N\}$. The set
$\mathfrak{F}$ is the union of a subsystem of the closed $\sigma$-discrete network
$\gamma_2$ (because so is each $F_i$). Recall that $\gamma_0 = f^{-1}(\gamma_2)$.
Since each $f^{-1} F_i$ is the union of a subsystem of $\gamma_0$, it follows that
so is $f^{-1}(\mathfrak{F})$. It is clear from the construction that
$f^{-1}(\mathfrak{F}) \cap F= \varnothing$. Let us show that
  $f^{-1}(\mathfrak{F})$ is closed in $X$.

  Indeed, take $x\in L_1$. Choose the least $k\in \mathbb N$ for which $x\in \Int(Q_k)$; we have
 $x\notin \Int(Q_i)$ for $i<k$. Recall that $Q_k=f^{-1}(P_k)$; therefore, $Q_k$ may intersect only
finitely many closed sets $f^{-1}(F_i)$, namely, only those with $1\leq i \leq k$. The remaining
  sets $f^{-1}(F_i)$ do not intersect $Q_k$. Therefore, the point $x\in L_1$ has an open neighborhood
  $Vx\subseteq \Int(Q_k)$ such that  $Vx\cap f^{-1}(\mathfrak{F})=\varnothing$ (see a similar argument in
  Construction~\ref{cons3.4}).

  If $x\notin(L_1 \cup f^{-1}(\mathfrak{F}))$, then the existence of an open neighborhood $Ux$ of $x$ for which
  $Ux\cap f^{-1}(\mathfrak{F})=\varnothing$ follows from the closedness of
  $\mathfrak{F}$ in $O_{n_0}S$ and the relations $\overline{\mathfrak{F}}\cap T =L=\overline{S} \cap T$.
  The required neighborhood $Ux$ of $x$ is the preimage of the corresponding neighborhood $Oy$ of the
point $y= f(x)$ in $M_0$. This proves the closedness of $f^{-1}(\mathfrak{F})$ in~$X$.
\end{construction}

 \begin{proof}[Proof of Theorem~\ref{th3.1}]
 Let us prove that  the closed $\sigma$-discrete network $\gamma_0 =
\bigcup\{\gamma^0_i: i=1,2,\dots \}$, where all $\gamma^0_i$ are discrete in $X$, constructed above
is intrinsically special.
We must show that, for any $F\in \gamma_0$ and any closed subset $\Phi$ of $X$ disjoint from
  $F$, there exists a subsystem $\gamma_0(F)$ of the network $\gamma_0$ such that the union
$\bigcup \gamma_0(F)$ is closed in $X$, $\Phi \subseteq \bigcup \gamma_0(F)$,  and
  $\bigcup \gamma_0(F) \cap F= \varnothing$.

 Construction~\ref{cons3.5} (see\ also Remark~\ref{rem3.3}) implies the existence of a closed
subset $\Phi_1$ of $X$ such that
  $\Phi \subseteq \Phi_1$, $\Phi_1 \cap F= \varnothing$ and $\Phi_1$ is the union of a
  subsystem $\gamma_0(F)$ of $\gamma_0$  in $X$, i.e., $\Phi_1$ = $\bigcup \gamma_0(F)$;
this is $\Phi_1 = f^{-1}(\mathfrak{F})$.

 This completes the proof of Theorem~\ref{th3.1}.
\end{proof}

\begin{theorem}
\label{th3.6}
Any stratifiable space $X$ is an S-space.
\end{theorem}

 \begin{proof}
According to Theorem~\ref{th3.1}, any stratifiable space has an intrinsically special network $\gamma$.
The author showed in \cite[Theorem~4]{15} that, in this case, $X$ is an everywhere f-space. In the author's
paper \cite[Assertion~2.5]{14}, the equivalence of the following conditions for any stratifiable space $X$
was proved:

\begin{enumerate}
\item[1.]
$X$ is an everywhere f-space;
\item[2.]
$X$ is an S-space.
\end{enumerate}

Thus, any stratifiable space is an S-space.

 This proves Theorem~\ref{th3.6}.
\end{proof}

 Theorem~\ref{th3.6} and Corollary~1 in the author's paper \cite{15} (cf.\ \cite[proof of Theorem~2.3, (2.1)]{14})
imply the following
theorem.

\begin{theorem}
\label{th3.7}
For a stratifiable space $X$, the following conditions are equivalent:
\begin{enumerate}
\item[\rm (a)]
$\dim X\leq n$\textup;
\item[\rm (b)]
$\Ind X \leq n$\textup;
\item[\rm (c)]
$X=\bigcup{X_i}$, $i=1,2,\dots n+1$, where each  $X_{i}$ is
a $G_\delta$-set in $X$ and $\dim{X_i}\leq 0$ for $i=1,2,\dots n+1$\textup;
\item[\rm (d)]
the space $X$ is  the image of a stratifiable space $Y$ with
$\dim Y=0$ under a perfect $(n+1)$-to-$1$ map.
\end{enumerate}
\end{theorem}

 A theorem similar to Corollary~1 in the author's paper \cite{15} was independently proved by S.~Oka
\cite[Theorem~4.3]{6} under the assumption that any stratifiable space is an S-space.

 Thus, for any stratifiable space $X$, we have $\dim X= \Ind X$.

\begin{corollary}
\label{cor3.8}
Let $X$ be a stratifiable space with a countable network. Then
$\dim X=\ind X=\Ind X$.
\end{corollary}

\begin{proof}
Any space with a countable network
is finally compact, and for finally compact spaces, the inequalities
$\dim X\leq \ind X\leq \Ind X$ hold \cite[Chap.~4, Sec.~8, Lemma~2]{1}.
\end{proof}

\begin{corollary}
\label{cor3.9}
Let $G$ be a stratifiable topological group. Then $\dim G =\Ind G$.
\end{corollary}

\begin{corollary}
\label{cor3.10}
Let $X$ be an $M_1$-space. Then $\dim X=\Ind X$.
\end{corollary}

Lemma~2.13 of \cite{14} and the coincidence of the dimensions $\dim$  and $\Ind$ for stratifiable spaces
imply the following statement.

\begin{corollary}
\label{cor3.11}
 Let $X$ and $Y$ be stratifiable spaces. Then $\Ind(X\times Y)\leq \Ind X+\Ind Y$.
\end{corollary}

   The following questions remain open.

\begin{question}[Arkhangel'skii \cite{2}]
Let $X$ be a space with a countable network being a topological group. Is it true that
$\dim X=\Ind X=\ind X$?
\end{question}

\begin{question}
Let $X$ be a normal Moore  space (a developable space). Is it true that
$\Ind X=\dim X$?
\end{question}

\begin{question}[Arkhangel'skii]
Let $X$ be a perfectly normal space for which all $X^n$, $n=1,2,\dots$, are perfectly normal as well.
Is it true that $\Ind X=\dim X$? Note that $X$ is not necessarily compact, i.e., the case of a noncompact
$X$ is also of interest.
\end{question}

\begin{question}
Does there exist in ZFC a stratifiable space $X$ for which $\ind X=0$ and $\Ind X\geq 2$?
\end{question}

\end{document}